\documentclass{amsart}

\usepackage{amssymb,color}
\usepackage{amsfonts}
\usepackage{amsmath}
\usepackage{euscript}
\usepackage{enumerate}
\usepackage{graphics}
\usepackage{pdfsync}
\newtheorem{theorem}{Theorem}[section]
\newtheorem{lemma}[theorem]{Lemma}

\newtheorem{prop}[theorem]{Proposition}

\newtheorem{exa}[theorem]{Example}

\newtheorem*{Theorem1'}{Theorem 1'}

\theoremstyle{definition}
\newtheorem{definition}[theorem]{Definition}

\theoremstyle{remark}

\numberwithin{equation}{section}

\renewcommand \ker{{\mathrm {ker}}}

\newcommand \Der{{\mathrm {Der}}}

\newcommand \GL{{\mathrm {GL}}}
\newcommand \gl{{\mathfrak {gl}}}

\newcommand \g{{\mathfrak {g}}}

\newcommand \h{{\mathfrak {h}}}

\newcommand \la{{\lambda}}

\newcommand \al{{\alpha}}
\newcommand \be{{\beta}}

\newcommand \ga{{\gamma}}

\renewcommand \sl{{\mathfrak {sl}}}

\newcommand \B{{\mathcal B}}
\newcommand \ad{{\mathrm {ad}}}

\begin{document}

\title[Free 2-step nilpotent Lie algebras and indecomposable modules]{Free 2-step nilpotent Lie algebras and indecomposable representations}

 \author{Leandro Cagliero}
\address{CIEM-CONICET, FAMAF-Universidad Nacional de C\'ordoba, Argentina.}
\email{cagliero@famaf.unc.edu.ar}
\thanks{This research was partially supported by grants from CONICET, FONCYT and SeCyT-UNC\'ordoba).}

\author{Luis Guti\'errez Frez}
\address{Instituto de Ciencias F\'{\i}sicas y Matem\'aticas, Universidad Austral de Chile, Chile}
\email{luis.gutierrez@uach.cl}

\author{Fernando Szechtman}
\address{Department of Mathematics and Statistics, Univeristy of Regina, Canada}
\email{fernando.szechtman@gmail.com}
\thanks{This research was partially supported by an NSERC grant}





\begin{abstract} Given an algebraically closed field $F$ of characteristic 0 and an $F$-vector space $V$,
let $L(V)=V\oplus\Lambda^2(V)$ denote the free 2-step nilpotent Lie algebra associated to $V$. In this paper, we classify all uniserial representations of the solvable Lie algebra $\g=\langle x\rangle\ltimes L(V)$, where $x$ acts on $V$ via an arbitrary invertible Jordan block.
\end{abstract}

\maketitle

\section{Introduction}
\label{intro}

We fix throughout an algebraically closed field $F$ of characteristic zero. All Lie
algebras and representations considered in this paper are assumed
to be finite dimensional over $F$, unless explicitly stated
otherwise.

According to \cite{M} (see also \cite{GP}), the task of classifying all indecomposable modules of an arbitrary Lie algebra is daunting.
However, in recent years there has been significant progress in classifying certain types of indecomposable modules for various
families of Lie algebras. See \cite{C1, C2, CGS, CPS, CS1, CS2, CS3, CMS, DdG, DR, J}, for example. The classification of all uniserial modules (those having a unique composition series) of distinguished classes of Lie algebras has been specially successful (see \cite{CPS,CGS,CS1,CS3}, for instance).

In this paper,  we make a further contribution in this direction by classifying all uniserial representations of the solvable Lie algebra $\g=\langle x\rangle\ltimes L(V)$,
where $V$ is a vector space, $L(V)=V\oplus\Lambda^2(V)$ is the free 2-step nilpotent Lie algebra associated to $V$,
and $x$ acts on $V$ via a single Jordan block $J_n(\la)$, with $\la\neq 0$. The case $n=1$, when $\Lambda^2(V)=0$, is covered in \cite{CS2},
so we will focus attention on the case $n>1$.

We say that a uniserial representation $R:\g\to\gl(U)$ is
\emph{relatively faithful} if $\ker(R)\cap\Lambda^2(V)$ is properly
contained in $\Lambda^2(V)$ and $\ker(R)\cap V=(0)$. It suffices
to consider the case when $R$ is relatively faithful, for if
$\Lambda^2(V)\subseteq \ker(R)$ then \cite{CPS} applies, if
$V\subseteq \ker(R)$ we may appeal to \cite{CS1}, and if
$(0)\neq \ker(R)\cap V\neq V$, we are led to consider a uniserial
representation $\overline{R}$ of $\langle \overline{x}\rangle\ltimes L(\overline{V})$,
where $\overline{V}$ is a factor of~$V$ by an $x$-invariant subspace,
$\overline{x}$ acts on $\overline{V}$ via an invertible Jordan block $J_m(\la)$, $1\leq m<n$,
and $\ker(\overline{R})\cap\overline{V}=(0)$.

Our main results are as follows.
In \S\ref{sec1} we define a family of relatively faithful uniserial representations of $\g$ (the case $\la=0$ being allowed).
Explicitly,  let $v_0,\dots,v_{n-1}$ be a basis of $V$ such that
$$
[x,v_0]=\la v_0+v_1, [x,v_1]=\la v_1+v_2,\dots, [x,v_{n-1}]=\la v_{n-1}.
$$
Given a triple $(a,b,c)$ of positive integers satisfying
$$
a+b=n+1,\; c\leq a\quad\text{ or }\quad c+b=n+1,\; a\leq c,
$$
two matrices $M\in M_{a\times b}$ and $N\in M_{b\times c}$ such that $$M_{a,1}\neq 0\text{ and }N_{b,1}\neq 0,$$
and a scalar $\al\in F$, we define a representation
$R=R_{a,b,c,M,N,\al}:\g\to\gl(d)$, $d=a+b+c$, in block form, in the following manner:
$$
R(x)=A=\left(
               \begin{array}{ccc}
                 J^a(\al) & 0 & 0 \\
                 0 & J^b(\al-\la) & 0 \\
                 0 & 0 & J^c(\al-2\la) \\
               \end{array}
             \right),
$$
where $J^p(\be)$ denotes the upper triangular Jordan block of size $p$ and eigenvalue~$\be$,
$$
R(v_k)=(\ad_{\gl(d)} A-\la 1_{\gl(d)})^k \left(
               \begin{array}{ccc}
                 0 & M & 0 \\
                 0 & 0 & N \\
                 0 & 0 & 0 \\
               \end{array}
             \right),\quad 0\leq k\leq n-1,
$$
$$
R(v\wedge w)=[R(v),R(w)],\quad v,w\in V.
$$

The representation~$R$ is always uniserial. It is also relatively faithful, except for an extreme case,
as described in Definition \ref{defeabc}. The length of~$R$, as defined in Definition \ref{dele}, is equal to 3 (it coincides
with the number of Jordan blocks of $R(x)$ in this case).

Conjugating all $R(y)$, $y\in \g$, by a suitable block diagonal matrix commuting with~$A$, one may normalize $R$,
in the sense of Definition \ref{defeabc}.
In \S\ref{sec3} we prove, for $\la\neq 0$, that every relatively faithful uniserial representation of $\g$
is isomorphic to one and only one normalized representation $R_{a,b,c,M,N,\al}$ of non-extreme type. This requires, in particular,
to prove that $\g$ has no relatively faithful uniserial representations of length $>3$. This is our most challenging obstacle, and it is proven in Theorem \ref{main1}. The ideas behind the proof of Theorem \ref{main1} are somewhat subtle
and are presented independently in \S\ref{secnueva}.

We are be very interested in knowing the classification of all uniserial modules of $\g$ when $\lambda=0$
(the case when $\g$ is nilpotent), but this seems to be a very difficult task.

In \S\ref{sec2} we determine when $R_{a,b,c,M,N,\al}$ is faithful (for arbitrary $\lambda$). It turns out that $R_{a,b,c,M,N,\al}$ is faithful
if and only if
$$
(a,b,c)\in\{(n,1,n),(n-1,2,n-1),(n,1,n-1),(n-1,1,n)\}.
$$
Sufficiency of this result is fairly delicate. Most of the work towards it is done in Proposition \ref{fieln-2}.
The case $n=3$ and $(a,b,c)=(2,2,2)$ is special, in the sense that it is the only faithful uniserial representation of $\g$
where all blocks are square (in this case of size~2). This case is intimately related to a representation of
the truncated current Lie algebra $\sl(2)\otimes F[t]/(t^3)$.

In \S\ref{secnew} we provide a generalization of our faithfulness result, stated without reference to Lie algebras
or their representations.

Our general notation, basic concepts and preliminary material can all be found in \S\ref{sec0}, \S\ref{sec1} and \S\ref{sec2}.

\section{The Lie algebra $\g$}\label{sec0}

We fix throughout a vector space $V$. There is a unique Lie algebra structure on
$$
L(V)=V\oplus\Lambda^2(V)
$$
such that
$$
[v,w]=v\wedge w,\quad v,w\in V
$$
and
$$
[u,v\wedge w]=0, \quad u,v,w\in V.
$$
The Lie algebra $L(V)$ is the \emph{free 2-step nilpotent Lie algebra associated to $V$}.
In particular we have the following straightforward lemma.

\begin{lemma}\label{2free} Let $\h$ be a Lie algebra and let $\Omega:V\to\h$ be a linear map satisfying
\begin{equation*}
[\Omega(V),[\Omega(V),\Omega(V)]]=0.
\end{equation*}
Then $\Omega$ has a unique extension to a homomorphism of Lie algebras $\Omega':L(V)\to\h$.
\end{lemma}

Given a Lie algebra $\h$ and a representation $\h\to\gl(V)$, we can make $\Lambda^2(V)$ into an $\h$-module via:
$$
x(v\wedge w)=xv\wedge w+v\wedge xw,\quad x\in\h, v,w\in V.
$$
This gives a representation $\h\to\gl(L(V))$ whose image we readily see to be in $\Der(L(V))$. This produces the Lie algebra
$$
\h\ltimes L(V).
$$
For the remainder of the paper we set
$$
\g=\langle x\rangle\ltimes L(V),
$$
where $x\in\gl(V)$.

\section{Relatively faithful uniserial representations of $\g$}\label{sec1}

Given $p\geq 1$ and $\al\in F$,
we write $J_p(\al)$ (resp. $J^p(\al)$) for the lower (resp. upper) triangular Jordan block of size $p$ and eigenvalue $\al$.

We suppose throughout this section that $\g=\langle x\rangle \ltimes L(V)$, where $x\in\gl(V)$ acts on $V$ via a single, lower triangular, Jordan block, say $J_n(\la)$ with $n>1$, relative to a basis
$v_0,\dots,v_{n-1}$ of $V$. The case $\la=0$ is allowed. Then $\g$ has the following defining relations:
\begin{equation}
\label{rela}
[v,w]=v\wedge w,\quad v,w\in V,
\end{equation}
\begin{equation}
\label{rela2}
[u,v\wedge w]=0, \quad u,v,w\in V,
\end{equation}
\begin{equation}
\label{rew}
[x,v_0]=\la v_0+v_1, [x,v_1]=\la v_1+v_2,\dots, [x,v_{n-1}]=\la v_{n-1}.
\end{equation}
We may translate (\ref{rew}) as
\begin{equation}
\label{rela3}
(\ad_\g\, x -\lambda 1_\g)^k v_0=v_k,\quad 0\leq k\leq n-1,
\end{equation}
and
\begin{equation}\label{rela4}
(\ad_\g x -\lambda 1_\g)^n v_0=0.
\end{equation}

\begin{definition}\label{dele}
Let $U$ be a non-zero $\g$-module.  Let $U_1$ be the subspace of $U$ annihilated by
$[\g,\g]$. Since $[\g,\g]$ is an ideal of $\g$, it is clear that $U_1$ is a $\g$-submodule of $U$. Moreover, since $[\g,\g]$ acts via nilpotent
operators on $U$, Engel's theorem ensures that $U_1\neq 0$. We then choose $U_2$ so that $U_2/U_1$ is the subspace of $U/U_1$ annihilated by $[\g,\g]$, and so on. This gives rise to a strictly increasing sequence of $\g$-submodules of $U$, namely
$$
0\subset U_1\subset U_2\subset\cdots\subset U_\ell=U.
$$
We define the \emph{length} of $U$ to be $\ell$.
Note that, since $\g$ is solvable and $F$ is algebraically closed,
the length of a Jordan-H\"older composition series of $U$ is $\dim U$.
\end{definition}

\begin{definition}\label{defeabc} Let $(a,b,c)$ be a triple of positive integers satisfying
\begin{equation}
\label{cond}
a+b=n+1,\; c\leq a\quad\text{ or }\quad c+b=n+1,\; a\leq c,
\end{equation}
let $M\in M_{a\times b}$, $N\in M_{b\times c}$ be such that $$M_{a,1}\neq 0\text{ and }N_{b,1}\neq 0,$$
and let $\al\in F$. Associated to this data we define a linear transformation
$R=R_{a,b,c,M,N,\al}:\g\to\gl(d)$, $d=a+b+c$, in block form, as follows:
\begin{equation}
\label{pet}
R(x)=A=\left(
               \begin{array}{ccc}
                 J^a(\al) & 0 & 0 \\
                 0 & J^b(\al-\la) & 0 \\
                 0 & 0 & J^c(\al-2\la) \\
               \end{array}
             \right),
\end{equation}
\begin{equation}
\label{pet2}
R(v_k)=(\ad_{\gl(d)} A-\la 1_{\gl(d)})^k \left(
               \begin{array}{ccc}
                 0 & M & 0 \\
                 0 & 0 & N \\
                 0 & 0 & 0 \\
               \end{array}
             \right),\quad 0\leq k\leq n-1,
\end{equation}
\begin{equation}
\label{pet3}
R(v\wedge w)=[R(v),R(w)],\quad v,w\in V.
\end{equation}

We refer to $M$ and $N$ as \emph{normalized}, if the last rows of $M$ and $N$ are equal to the first
canonical vectors of $F^b$ and $F^c$, respectively, and the first column of~$M$ is equal to the last canonical
vector of $F^a$. In this case, we say that $R$ itself is \emph{normalized}.
If $R$ is normalized, we say that $R$ is of \emph{extreme type} if $n$ is odd,
$a=1$, $c=1$ and $N_{i,1}=0$ for all even $i$.
\end{definition}

Conjugating all
$R(y)$, $y\in \g$, by a suitable block diagonal matrix commuting with~$A$, it is always possible to normalize $R$,
as seen in \cite[Lemma 2.5]{CPS}.

\begin{prop}\label{normali} The linear map $R_{a,b,c,M,N,\al}$ is a uniserial
representation.
\end{prop}

\begin{proof} It follows from Lemma \ref{2free} that (\ref{pet2})-(\ref{pet3}) define a Lie homomorphism $L(V)\to\gl(d)$.
By (\ref{cond}), we have $a+b\leq n+1$ and $b+c\leq n+1$, so \cite[Proposition 2.2]{CPS} ensures that the relations (\ref{rela3}) and (\ref{rela4})
are preserved, whence $R$ is a representation. Since $M_{a,1}\neq 0$ and $N_{b,1}\neq 0$, $R$ is clearly uniserial.
\end{proof}

\begin{prop}\label{normal2} Assume $\lambda\ne0$. The normalized representations $R_{a,b,c,M,N,\al}$ are non-isomorphic to each other.
The normalized representation $R_{a,b,c,M,N,\al}$ is relatively faithful, except only for the extreme type.
\end{prop}

\begin{proof} Considering the eigenvalues of the image of $x$ as well as their multiplicities, the only possible isomorphisms are easily seen to be between $R_{a,b,c,M,N,\al}$ and $R_{a,b,c,M',N',\al}$. Suppose $T\in\GL(d)$, $d=a+b+c$, satisfies
$$
T R_{a,b,c,M,N,\al}(y) T^{-1}=R_{a,b,c,M',N',\al}(y),\quad y\in\g.
$$
Then $T$ commutes with $R_{a,b,c,M,N,\al}(x)=J^a(\al)\oplus J^{b}(\al-\la)\oplus J^c(\al-2\la)$, and therefore $T=T_1\oplus T_2\oplus T_3$,
where $T_1,T_2,T_3$ are polynomials in $J^a(0),J^b(0),J^c(0)$, respectively, with non-zero constant term.
This means that every superdiagonal
of~$T_i$, $1\leq i\leq 3$, has equal entries. Using this feature of $T_1,T_2,T_3$ in
$$
T R_{a,b,c,M,N,\al}(v_0) =R_{a,b,c,M',N',\al}(v_0)T
$$
together with the fact that $M,N$ and $M',N'$ are normalized, we readily find that~$T$ is a scalar operator, whence $M=M'$ and $N=N'$.

Since $a+b=n+1$ or $b+c=n+1$, \cite[Proposition 2.2]{CPS} yields $\ker(R)\cap V=(0)$.
It remains to determine when
is $\Lambda^2(V)\subseteq \ker(R)$. By \cite[Theorem 3.2]{CPS}, this can only happen when $n$ is odd, $a=1$, $c=1$,
in which case direct computation forces $N_{i,1}=0$ for all even $i$.
\end{proof}

\section{Determining the faithful uniserial representations of $\g$}\label{sec2}

We assume throughout this section that $\g=\langle x\rangle \ltimes L(V)$, where $x$ acts on $V$ via a single lower Jordan
block $J_n(\la)$, $n>1$, relative to a basis $v_0,\dots,v_{n-1}$ of $V$.

\begin{definition} Given a sequence $(d_1,\dots,d_\ell)$ of positive integers, we view
every $M\in M_d$,  for $d=d_1+\cdots+d_\ell$, as partitioned into $\ell^2$ blocks
$M(i,j)\in M_{d_1\times d_j}$, $1\leq i,j\leq \ell$. For $0\leq i\leq \ell-1$,
by the $i$th superdiagonal of $M$ we mean the blocks $M(1,1+i),M(2,2+i),\dots,M(\ell-i,\ell)$,
and we say that $M$ is an $i$-diagonal block
matrix if all other blocks of $M$ are equal to 0.
We refer to $M$ as block upper triangular if $M(i,j)=0$ for all $i>j$ and as block
strictly upper triangular if $M(i,j)=0$ for all $i\geq j$.
\end{definition}

\begin{definition} Given an integer $\ell>2$, a sequence of positive integers $(d_1,\dots,d_\ell)$, and a scalar $\al\in F$,
a representation $R:\g\to\gl(d)$ is said to be \emph{standard} relative to $(\ell,(d_1,\dots,d_\ell),\al)$
if the following conditions hold:
$$d_1+\cdots+d_\ell=d;\quad d_{i}+d_{i+1}\leq n+1\text{ for all }i;$$
$R(x)$ is the 0-diagonal block matrix
$$
A=J^{d_1}(\alpha)\oplus J^{d_2}(\alpha-\la)\oplus\cdots\oplus J^{d_\ell}(\alpha-(\ell-1)\la);
$$
every $R(v)$, $v\in V$, is a 1-diagonal block matrix; every block in the first superdiagonal of $R(v_0)$
has non-zero bottom left entry.

Let $M_1,\dots,M_{\ell-1}$ denote the blocks in the first superdiagonal of $R(v_0)$.
We say that $R$ is \emph{normalized standard} relative to $(\ell,(d_1,\dots,d_\ell),\al)$ if,
in addition to the above conditions, the last row of each $M_i$ is equal to the first canonical vector,
and the first column of $M_{1}$ is the last canonical vector.
\end{definition}

Note that a standard representation $R$ is always uniserial, and its length,
as defined in Definition \ref{dele}, is equal to $\ell$.
Observe also that if $R$ is a standard representation then every $R(v\wedge w)$, $v,w\in V$,
is a 2-diagonal block matrix.

\begin{lemma}\label{funk} Given an integer $\ell>2$, a sequence of positive integers $(d_1,\dots,d_\ell)$, and a scalar $\al\in F$,
let $R:\g\to\gl(d)$ be a standard representation relative to them. Then $\ker(R)\cap V=(0)$
if and only if $d_{i}+d_{i+1}=n+1$ for at least one $i$.
\end{lemma}

\begin{proof} Since the $x$-invariant subspaces of $V$ form a chain, we have $\ker(R)\cap V=(0)$ if and only if $v_{n-1}\notin\ker(R)$,
which is equivalent to $d_{i}+d_{i+1}=n+1$ for some $i$, by \cite[Proposition 2.2]{CPS}.
\end{proof}

\begin{lemma}\label{duy} Given an integer $\ell>2$, a sequence of positive integers $(d_1,\dots,d_\ell)$,
and a scalar $\al\in F$,
let $R:\g\to\gl(d)$ be a standard (resp. normalized standard) representation relative to them.
Then the dual representation is similar to a representation $T:\g\to\gl(d)$ that is standard
(resp. normalized standard) relative to $(\ell, (d_\ell,\dots,d_1),(\ell-1)\la-\al)$.
Moreover, $R$ is faithful (resp. relatively faithful) if and only so is $T$.
\end{lemma}

\begin{proof} This is straightforward.
\end{proof}


\begin{prop}\label{fieln-2}  Given an integer $n\geq 2$, let $(p_1,\dots,p_{n-1}),(q_1,\dots,q_{n-1})\in F^{n-1}$
be such that $p_j+q_j\ne0$ for all $j$, and
let $z,w\in F$ be non-zero.
Associated to these data, we consider matrices
$$
P_0,\dots,P_{n-1}\in M_{n-1\times 2},\qquad Q_0,\dots,Q_{n-1}\in M_{2\times n-1},
$$
having the following structure:
\[\small
 \begin{array}{lll}
P_0=\left(
      \begin{array}{cc}
        * & * \\
        \vdots & \vdots \\
        * & * \\
        z & * \\
      \end{array}
    \right),  &
    P_1=\left(
      \begin{array}{cc}
        * & * \\
        \vdots & \vdots \\
        * & * \\
        z & * \\
        0 & -p_1z \\
      \end{array}
    \right), &
P_2=\left(
      \begin{array}{cc}
        * & * \\
        \vdots & \vdots \\
        * & * \\
        z & * \\
        0 & -p_2z \\
        0 & 0 \\
      \end{array}
    \right), \\
P_3=\left(
      \begin{array}{cc}
        * & * \\
        \vdots & \vdots \\
        * & * \\
        z & * \\
        0 & -p_3z \\
        0 & 0 \\
        0 & 0 \\
      \end{array}
    \right),\dots, &
P_{n-2}=\left(
      \begin{array}{cc}
        z & * \\
        0 & -p_{n-2}z\\
        0 & 0 \\
      \vdots & \vdots \\
        0 & 0 \\
      \end{array}
    \right), &
    P_{n-1}=\left(
      \begin{array}{cc}
        0 & -p_{n-1}z\\
        0 & 0 \\
      \vdots & \vdots \\
        0 & 0 \\
      \end{array}
    \right),
 \end{array}
\]
\medskip
\[\small
 \begin{array}{ll}
Q_{0}= \left(
             \begin{array}{cccc}
               * & * & \dots & * \\
               w & * & \dots & * \\
             \end{array}
           \right), &
           Q_{1}= \left(
             \begin{array}{ccccc}
               q_1w & * & * & \dots & * \\
               0 & -w & * & \dots & * \\
             \end{array}
           \right), \\[5mm]
Q_2= \left(
             \begin{array}{cccccc}
               0 & -q_2w & * & *& \dots & * \\
               0 & 0 & w & * &\dots & * \\
             \end{array}
           \right), &
            Q_{3}= \left(
             \begin{array}{cccccc}
               0 & 0 & q_3w & * & \dots & * \\
               0 & 0 & 0 & -w & \dots & * \\
             \end{array}
           \right),\dots, \\[5mm]
Q_{n-2}= \left(
             \begin{array}{cccc}
               0\;  \dots\; 0 & (-1)^{n-3}q_{n-2}w & * \\
               0\;  \dots\; 0 & 0 & (-1)^{n-2}w \\
             \end{array}
           \right), &
Q_{n-1}=\left(
             \begin{array}{cc}
               0\;  \dots\;  0 & (-1)^{n-2}q_{n-1}w \\
               0\;  \dots\;  0 &  0 \\
             \end{array}
           \right).        \\[5mm]
 \end{array}
\]
Then the matrices $T_{i,j}\in M_{n-1}$, $0\leq i<j\leq n-1$, defined by
$$
T_{i,j}=P_iQ_j-P_jQ_i, \quad 0\leq i<j\leq n-1,
$$
are linearly independent.
\end{prop}

\begin{proof} By induction on $n$. In the base case $n=2$, we have
$$
P_0=\begin{pmatrix}
     z & *
    \end{pmatrix}, \;
    P_1=\begin{pmatrix}
     0 & -p_1z
    \end{pmatrix},\;
    Q_0=\left(
                   \begin{array}{c}
                     * \\
                     w \\
                   \end{array}
                 \right), Q_1=\left(
                   \begin{array}{c}
                     q_1w \\
                     0 \\
                   \end{array}
                 \right).
$$
Therefore
$$
T_{0,1}=\begin{pmatrix}(p_1+q_1)wz\end{pmatrix}\neq 0.
$$
Assume that $n>2$ and the that result is true for $m=n-1$. Let
\begin{equation*}
\mathcal{T}=\underset{0\leq i<j\leq n-1}\sum\al_{i,j} T_{i,j}
\end{equation*}
and assume $\mathcal{T}=0$. We wish to show that
\begin{equation}
\label{alcero-2}
\al_{i,j}=0,\quad 0\leq i<j\leq n-1.
\end{equation}
It suffices to show that
\begin{equation}
\label{alcero}
\al_{0,j}=0,\quad 1\leq j\leq n-1.
\end{equation}
Indeed, assume we have proven (\ref{alcero}). Since $\mathcal{T}=0$, we obtain
\begin{equation}
\label{alcero2}
\underset{1\leq i<j\leq n-1}\sum\al_{i,j} T_{i,j}=0.
\end{equation}
Let $P'_0,\dots,P'_{m-1}\in M_{m-1\times 2}$ and $Q'_0,\dots,Q'_{m-1}\in M_{2\times m-1}$
be the matrices
obtained by deleting the last rows of $P_1,\dots,P_{n-1}$ and the first columns of $Q_1,\dots,Q_{n-1}$,
and let $T'_{i,j}=P'_iQ'_j-P'_jQ'_i$, $0\leq i<j\leq m-1$.
It follows automatically from (\ref{alcero2}) that
\begin{equation*}
\underset{0\leq i<j\leq m-1}\sum\al'_{i,j} T'_{i,j}=0,
\end{equation*}
where $\al'_{i,j}=\al_{i+1,j+1}$ and, from the inductive hypothesis, we conclude
\begin{equation}
\label{alcero4}
\al_{i,j}=0,\quad 1\leq i<j\leq n-1.
\end{equation}
We may now obtain (\ref{alcero-2}) from (\ref{alcero}) and (\ref{alcero4}).

We proceed to prove (\ref{alcero}). In fact we will prove by induction on $k\le n-1$ that $\alpha_{i,j}=0$ whenever $i<j$ and $i+j\leq k$.

The base case $k=1$ is straightforward. Indeed, from $\mathcal{T}_{n-1,1}=\alpha_{0,1}(p_1+q_1)wz$, infer $\alpha_{0,1}=0$.

Suppose $1<k\leq n-1$ and assume that $\alpha_{i,j}=0$ whenever $i<j$ and $i+j\leq k-1$. Using this, a direct computation reveals that, for $i-j=n-1-k$, we have
\[
 \mathcal{T}_{i,j}=\begin{cases}
(-1)^j(\alpha_{j,k-j}\,q_{j}-\alpha_{j-1,k+1-j}\,p_{k+1-j})\,wz, & \text{if  $1\le j<\frac{k}2$;} \\[2mm]
-(-1)^{\frac{k}{2}}\,\alpha_{\frac{k}{2}-1,\frac{k}{2}+1}\,p_{\frac{k}{2}+1}\,wz, & \text{if $j=\frac{k}2$;} \\[2mm]
-(-1)^{\frac{k+1}{2}}\,\alpha_{\frac{k-1}{2},\frac{k+1}{2}}\big(q_{\frac{k+1}{2}}+p_{\frac{k+1}{2}}\big)\,wz, & \text{if $j=\frac{k+1}2$;} \\[2mm]
-(-1)^{\frac{k+2}{2}}\,\alpha_{\frac{k}{2}-1,\frac{k}{2}+1}\,q_{\frac{k}{2}+1}\,wz, & \text{if $j=\frac{k}2+1$;} \\[2mm]
-(-1)^j(\alpha_{k-j,j}\,q_{j}-\alpha_{k+1-j,j-1}\,p_{k+1-j})\,wz, & \text{if  $\frac{k}2+1<j\le n-1$;}
                   \end{cases}
\]
that is
\[
 \begin{array}{cc|cl}
  i&j & \mathcal{T}_{i,j}/wz \\
  \hline & \\[-2mm]
  n-k & 1  & -\alpha_{1,k-1}q_{1}+\alpha_{0,k}p_{k} \\[2mm]
  n-k+1 & 2  & \alpha_{2,k-2}q_{2}-\alpha_{1,k-1}p_{k-1} \\[2mm]
  n-k+2 & 3  & -\alpha_{3,k-3}q_{3}+\alpha_{2,k-2}p_{k-2} \\[1mm]
  \vdots &  \vdots &  \vdots  \\[1mm]
  n-1-\frac{k}{2} & \frac{k}{2}  & -(-1)^{\frac{k}{2}}\,\alpha_{\frac{k}{2}-1,\frac{k}{2}+1}\,p_{\frac{k}{2}+1} &\text{(if $k$ is even)}\\[2mm]
  n-1-\frac{k-1}{2} & \frac{k+1}{2}  & -(-1)^{\frac{k+1}{2}}\,\alpha_{\frac{k-1}{2},\frac{k+1}{2}}\big(q_{\frac{k+1}{2}}+p_{\frac{k+1}{2}}\big)&\text{(if $k$ is odd)} \\[2mm]
 n-1-\frac{k-2}{2} &  \frac{k+2}{2}  &   -(-1)^{\frac{k+2}{2}}\,\alpha_{\frac{k}{2}-1,\frac{k}{2}+1}\,q_{\frac{k}{2}+1} & \text{(if $k$ is even)}  \\[1mm]
   \vdots &  \vdots &  \vdots  \\[1mm]
  n-3 & k-2  & -(-1)^{k-2}(\alpha_{2,k-2}\,q_{k-2}-\alpha_{3,k-3}\,p_{3}) \\[2mm]
  n-2 & k-1  & -(-1)^{k-1}(\alpha_{1,k-1}\,q_{k-1}-\alpha_{2,k-2}\,p_{2}) \\[2mm]
  n-1 & k  & -(-1)^k(\alpha_{0,k}\,q_{k}-\alpha_{1,k-1}\,p_{1}) \\
  \end{array}
\]
Since, by hypothesis, $p_j+q_j\ne0$ for all $j$
(which in turns implies that either $p_j$ or $q_j$ is non-zero for all $j$) we obtain that  (\ref{alcero}) holds.
\end{proof}

\begin{theorem}\label{main0} A representation $R_{a,b,c,M,N,\al}$ of $\g$ is faithful if and only if
\begin{equation}
\label{conabc}
(a,b,c)\in\{(n,1,n),(n-1,2,n-1),(n,1,n-1),(n-1,1,n)\}.
\end{equation}
\end{theorem}

\begin{proof} We divide the proof into two parts.

{\sc Necessity.} Suppose the representation $R=R_{a,b,c,M,N,\al}:\g\to\gl(d)$ is faithful, where $d=a+b+c$.

Let $S$ be the subspace of $\gl(d)$ of all matrices
$$
\left(
               \begin{array}{ccc}
                 0 & 0 & P \\
                 0 & 0 & 0 \\
                 0 & 0 & 0 \\
               \end{array}
             \right),\quad P\in M_{a\times c}.
$$
Letting $A$ be as in (\ref{pet}), we view $S$ as an $F[t]$-module via $\ad_{\gl(d)}A-2\la 1_{\gl(d)}$.
As in \cite[Proposition 2.1]{CPS}, we see that $\ad_{\gl(d)}A-2\la 1_{\gl(d)}$ acts nilpotently on $S$ with nilpotency degree $a+c-1$.
On the other hand, we may view $\Lambda^2(V)$ as an $F[t]$-module via $\ad_\g x -2\la 1_\g$. Direct computation (alternatively, we may use the theory
of $\sl(2)$-modules), reveals that $\ad_\g x -2\la 1_\g$ acts on $\Lambda^2(V)$ with nilpotency degree $2n-3$. Indeed, we have
\begin{equation}
\label{forme}
(\ad_\g x -2\la 1_\g)^m(v\wedge w)=\underset{i+j=m}\sum {{m}\choose {i}} (x-\la 1_V)^i v\wedge (x-\la 1_V)^j w.
\end{equation}
Set $m=2n-3$ in (\ref{forme}) and take $v=v_p$ and $w=v_q$ with $0\leq p<q\leq n-1$. Then the right hand side of (\ref{forme}) is equal to 0
(including the extreme case $p=0,q=1$, which produces ${{2n-3}\choose {n-1}} v_{n-1}\wedge v_{n-1}=0$). Next set $m=2n-4$ in (\ref{forme}) and take $v=v_0$ and $w=v_1$. Then the right hand side of (\ref{forme}) is equal to
$$
\left [{{2n-4}\choose {n-1}}-{{2n-4}\choose {n-2}}\right ]v_{n-1}\wedge v_{n-2}\neq 0.
$$
Since $R$ is faithful, restricting $R$ to $\Lambda^2(V)$ yields a linear monomorphism $T:\Lambda^2(V)\to S$. It follows from \cite[Lemma 3.1]{CPS}
that $T$ commutes with the indicated actions of $F[t]$, so that $T$ is a monomorphism of $F[t]$-modules. It follows from above that
\begin{equation}
\label{ace}
2n-3\leq a+c-1.
\end{equation}

On the other hand, by (\ref{cond}), we have $a+b=n+1$ or $c+b=n+1$.
By duality (see Lemma \ref{duy}), we may
assume that $a+b=n+1$.
Suppose, if possible, that $b+c<n$. As the $x$-invariant subspaces of $V$ form a chain, it follows from \cite[Proposition 2.2]{CPS} that blocks (2,3) of $R(v_{n-1})$ and $R(v_{n-2})$ are equal to 0 (alternatively, appeal to a direct computation based on (\ref{pet}) and (\ref{pet2})).
Then (\ref{pet3}) yields $R(v_{n-2}\wedge v_{n-1})=0$, a contradiction. We infer $b+c\geq n$. It follows from (\ref{cond}) that $b+c=n$ or $b+c=n+1$. In the second case $c=a$, so (\ref{ace}) gives $a\geq n-1$, whence
$$
(a,b,c)\in\{(n,1,n),(n-1,2,n-1)\}.
$$
In the first case $c=a-1$, so (\ref{ace}) gives $a\geq n-\frac{1}{2}$, whence $(a,b,c)=(n,1,n-1)$.

{\sc Sufficiency.} We wish to show that $R=R_{a,b,c,M,N,\al}$ is faithful whenever (\ref{conabc}) holds.

By duality (see Lemma \ref{duy}), we may restrict to the cases
\begin{equation}
\label{set_abc}
(a,b,c)\in\{(n,1,n),(n-1,2,n-1),(n,1,n-1)\}.
\end{equation}

We will write $P(y),Q(y)$ and $T(y)$ for blocks $(1,2),(2,3)$ and (1,3) of $R(y)$, $y\in\g$, respectively.

By Proposition \ref{normal2}, $R$ is relatively faithful (it follows from \eqref{set_abc} that,
after normalizing $R$, we are not in the extreme case)
and thus $R$ is faithful if and only if the matrices $T(v_i\wedge v_j)$, $0\leq i<j\leq n-1$, are linearly independent.

$\bullet (a,b,c)=(n-1,2,n-1)$.
Set $(p_1,\dots,p_{n-1})=(q_1,\dots,q_{n-1})=(1,\dots,n-1)$ and,
for $i=0,\dots,n-1$, let $P_i=P(v_i)\in M_{n-1\times 2}$ and $Q_i=Q(v_i)\in M_{2\times n-1}$.
It is not difficult to see that these vectors and matrices satisfy the hypothesis of
Proposition \ref{fieln-2} and thus, considering \eqref{pet3}, we obtain that
$$T(v_i\wedge v_j)=P(v_i)Q(v_j)-P(v_j)Q(v_i)=P_iQ_j-P_jQ_i=T_{i,j}, \quad 0\leq i<j\leq n-1,$$
are linearly independent.

$\bullet $ $(a,b,c)=(n,1,n)$.  Note that $T(v_i\wedge v_j)$, $0\leq i<j\leq n-1$, form the canonical
basis of the space $\mathfrak{so}(n)$ of all $n\times n$ skew-symmetric matrices.

$\bullet $ $(a,b,c)=(n,1,n-1)$.  Again, $T(v_i\wedge v_j)$, $0\leq i<j\leq n-2$, form the canonical
basis of $\mathfrak{so}(n-1)$, viewed as the subspace of $\mathfrak{so}(n)$ of matrices with zero first row and last column. On the other hand,
noting that $Q(v_{n-1})=0$, we see that $T(v_i\wedge v_{n-1})$, $0\leq i<n-1$, form the (opposite of the) canonical basis of $F^{n-1}$,
viewed as top left corner, say $C$, of $M_n$. Since $\mathfrak{so}(n-1)\cap C=(0)$, the result follows.
\end{proof}


\begin{exa}\label{sl2lindo}{\rm An interesting example occurs when $n=3$ and $a=b=c=2$. Then we do get
a faithful module above of a very special nature: it is the
only faithful uniserial module of $\g$ where all the blocks are squares.
Take $\la=\al=0$ (the other cases are easy modifications).

Given a Lie algebra $L$ and an associative commutative algebra $A$, we know that $L\otimes A$ is  Lie algebra under
$[x\otimes a,y\otimes b]=[x,y]\otimes ab$. Moreover, if $R_1:L\to\gl(V_1)$ and $R_2:A\to\gl(V_2)$ are representations,
then $R_1\otimes R_2:L\otimes A\to\gl(V\otimes A)$ is a representation.

Now take $L=\sl(2)$, with standard basis $E,H,F$, and $A=F[t]/(t^3)$.
Let $R_1$ be the irreducible representation of highest weight 1 and
let $R_2$ be the regular representation.
If we restrict the representation $R_1\otimes R_2$ to the subalgebra of $\sl(2)\otimes F[t]/(t^3)$
generated by $\{E\otimes 1, F\otimes t\}$ (which is isomorphic to $\g$)
we obtain the case $n=3$ and $a=b=c=2$ of the above construction.
}
\end{exa}

\section{Faithfulness in purely matrix terms}\label{secnew}

The following general version of Theorem \ref{main0} is stated in purely matrix terms. Given integers $a,b\geq  1$, let $\Phi_{a,b}:M_{a\times b}\to M_{a\times b}$ be the nilpotent linear operator defined by
$$
\Phi_{a,b}(X)=J^a(0)X-XJ^b(0).
$$
We will write $\Phi$ instead of $\Phi_{a,b}$ when no confusion is possible.

\begin{theorem}\label{lidep}
Given a triple $(a,b,c)$ of positive integers and
a pair $(P,Q)$ of matrices such that $P\in M_{a\times b}$, $Q\in M_{b\times c}$, we
define the matrices $P_i$, $Q_i$, $T_{i,j}$ by
$$
P_i=\Phi^i(P),\; Q_i=\Phi^i(Q),\quad i\geq 0,
$$
$$
T_{i,j}=P_iQ_j-P_jQ_i,\quad 0\leq i<j,
$$
and set
$$
n=\max\{a+b-1,b+c-1\}.
$$
Then $P_i=Q_i=0$ for $i\ge n$ and the set  $\mathcal{T}=\{T_{i,j}:0\leq i<j\leq n-1\}$ is linearly independent if and only if
exactly one of the following three conditions hold:
$$
P_{a,1}\neq 0, Q_{b,1}\neq 0\text{ and }(a,b,c)\in\{(n,1,n),(n-1,2,n-1),(n,1,n-1),(n-1,1,n)\},
$$
$$
P_{a,1}=0, P_{a-1,1}\neq 0, Q_{b,1}\neq 0\text{ and } (a,b,c)=(n,1,n),
$$
$$
P_{a,1}\neq 0, Q_{b,1}=0, Q_{b,2}\neq 0\text{ and } (a,b,c)=(n,1,n).
$$
\end{theorem}

\begin{proof} The case $n=1$ is obvious, so we assume $n>1$.

It follows from \cite[Proposition 2.2]{CPS} that $P_{i}=Q_{i}=0$ for $i\ge n$.
If $P_{a,1}=0$ and $Q_{b,1}=0$ then \cite[Proposition 2.1]{CPS} implies $P_{n-1}=Q_{n-1}=0$
and thus $\mathcal{T}$ is linearly dependent.

For the remainder of the proof we assume that $P_{a,1}\neq 0$ or $Q_{b,1}\neq 0$. Three cases arise.

\medskip

Case 1: $P_{a,1}\neq 0$ and $Q_{b,1}\neq 0$. By Theorem \ref{main0}, the set $\mathcal{T}$ is
linearly independent if and only if $(a,b,c)\in\{(n,1,n),(n-1,2,n-1),(n,1,n-1),(n-1,1,n)\}$.

\medskip

Case 2:  $P_{a,1}=0$ and $Q_{b,1}\neq 0$. Suppose first $\mathcal{T}$ linearly independent.
The necessity part of the proof of Theorem \ref{main0} still implies that $(a,b,c)$ belongs to
$\{(n,1,n),(n-1,2,n-1),(n,1,n-1),(n-1,1,n)\}$. We will show that $P_{a-1,1}\neq 0$ and
$(a,b,c)=(n,1,n)$.

The fact that $P_{a,1}=0$ and \cite[Proposition 2.1]{CPS} imply that $P_{n-1}=0$.
If $b+c<n+1$ then $Q_{n-1}=0$, by \cite[Proposition 2.2]{CPS},
so $T_{i,n-1}=0$ for all $0\leq i<n-1$, a contradiction.
Thus $b+c=n+1$. Since $P_{a,1}=0$, every entry of $P_{n-2}$, except
perhaps for its top right entry, is equal to 0. By construction, $Q_{n-1}$ shares this property. Since
$$
T_{n-2,n-1}=P_{n-2}Q_{n-1}-P_{n-1}Q_{n-2}=P_{n-2}Q_{n-1}\neq 0,
$$
we infer $b=1$ and thus $c=n$.
Moreover, if $a<n$ then $b=1$, $P_{a,1}=0$ and  \cite[Proposition 2.1]{CPS} imply
$P_{n-2}=0$, so $T_{n-2,n-1}=0$, a contradiction.
Therefore $a=n$. Finally, if $P_{n-1,1}=0$ we obtain again $P_{n-2}=0$. Thus $P_{n-1,1}\neq 0$.

Finally, suppose  $(a,b,c)=(n,1,n)$ and $P_{a-1,1}\neq 0$. By deleting the last row of $P$
and arguing as in Case 1 for $(a',b',c')=(n-1,1,n)$, we obtain that $\mathcal{T}$ is linearly independent.

Case 3:  $P_{a,1}\neq 0$ and $Q_{b,1}=0$. This is completely analogous to Case 2.
\end{proof}

\section{Lemmata}\label{secnueva}

Recall the meaning of $\Phi$ given in \S\ref{secnew}.

\begin{lemma}\label{kernel}
Let $Y\in M_{a,b}$. Then $\Phi(Y)=0$ if and only if
\begin{equation}\label{aleb}
  Y=\begin{pmatrix}
 0 &\cdots &0 &  \nu_1& \nu_2  &\cdots & \nu_a  \\
  0 &\cdots &0 & 0 & \nu_1 &\ddots & \vdots   \\
  \vdots & &\vdots & \vdots &\vdots &\ddots&  \nu_2    \\
  0 &\cdots &0 & 0 &   0   &\cdots &\nu_1
 \end{pmatrix},\text{ if }a\leq b,
\end{equation}
\begin{equation}\label{blea}
 Y=
 \begin{pmatrix}
 \mu_1& \mu_2  &\cdots & \mu_b  \\
 0 & \mu_1 &\ddots & \vdots   \\
 \vdots &\vdots &\ddots&  \mu_2    \\
 0 & 0   &\cdots &\mu_1 \\
 0 & \cdots  &\cdots & 0\\
 \vdots & \cdots  &\cdots &\vdots \\
 0 &\cdots  &\cdots &0\\
 \end{pmatrix},\text{ if }b\leq a
\end{equation}
for some $\mu_i,\nu_i\in F$.
\end{lemma}

\begin{proof} View $M_{a,b}$ as an $\sl(2)$-module as in the proof of \cite[Proposition 2.1]{CPS}.
The nullity of $\Phi$ is the number $m=\min\{a,b\}$ of irreducible $\sl(2)$-submodules of $M_{a,b}$. On the other
hand, if $m=a$ (resp. $m=b$) we readily verify that $Y$ as in (\ref{aleb}) (resp. (\ref{blea})) satisfies $\Phi(Y)=0$.
\end{proof}





We say that $X\in M_{a,b}$ is a lowest matrix if $X_{a,1}=1$.

\begin{lemma}\label{lemma1}
Let $X_1\in M_{a,b_1}$, $X_2\in M_{b_2,c}$, $Y_1\in M_{b_1,c}$ and $Y_2\in M_{a,b_2}$.
Assume that $X_1$ and $X_2$ are lowest matrices, that
$$
(Y_1,Y_2)\ne(0,0),\;\Phi(Y_1)=0,\Phi(Y_2)=0,
$$
and set
\[
Z=X_1Y_1-Y_2X_2.
\]
If $Z=0$ then $a\le b_2$,  $c\le b_1$ and
 \begin{equation}\label{eq.Z_0}
  Y_2=\begin{pmatrix}
 0 &\cdots &0 &  \nu_1& \nu_2  &\cdots & \nu_a  \\
  0 &\cdots &0 & 0 & \nu_1 &\ddots & \vdots   \\
  \vdots & &\vdots & \vdots &\vdots &\ddots&  \nu_2    \\
  0 &\cdots &0 & 0 &   0   &\cdots &\nu_1
 \end{pmatrix},\qquad
 Y_1=
 \begin{pmatrix}
 \mu_1& \mu_2  &\cdots & \mu_c  \\
 0 & \mu_1 &\ddots & \vdots   \\
 \vdots &\vdots &\ddots&  \mu_2    \\
 0 & 0   &\cdots &\mu_1 \\
 0 & \cdots  &\cdots & 0\\
 \vdots & \cdots  &\cdots &\vdots \\
 0 &\cdots  &\cdots &0\\
 \end{pmatrix},
 \end{equation}
with $\mu_1=\nu_1\ne0$.
\end{lemma}
\begin{proof} If $Y_1\neq 0$, let $C_i$, $1\leq i\leq c$, be the first column of $Y_1$ that is non-zero.
By Lemma \ref{kernel}, we have
$$
C_i=\left(
  \begin{array}{c}
    \mu \\
    0 \\
    \vdots \\
    0 \\
  \end{array}
\right),\quad \mu\neq 0.
$$
Since $X_1$ is a lowest matrix, it follows that column $i$ of $X_1Y_1$ is equal to
$$
\left(
  \begin{array}{c}
    * \\
    \vdots \\
    *\\
    \mu \\
  \end{array}
\right)\quad \mu\neq 0.
$$
If $Y_2\neq 0$, let $R_j$, $1\leq j\leq b_2$, be the last row of $Y_2$ that is non-zero.
By Lemma~\ref{kernel}, we have
$$
R_j=(0,\dots,0,\nu),\quad\nu\neq 0.
$$
Since $X_2$ is a lowest matrix, it follows that row $j$ of $Y_2X_2$ is equal to
$$
(\nu,*,\dots,*),\quad\nu\neq 0.
$$

Since $(Y_1,Y_2)\neq 0$ and $Z=0$, we infer from above that $Y_1\neq 0$ and $Y_2\neq 0$. If either
if $a>b_2$ or $Y_2$ does not have full rank, then Lemma \ref{kernel} implies that the last row of $Y_2$ is 0, so by above $Z_{a,i}=\mu$, a contradiction.
Similarly, if either $c>b_1$ or $Y_1$ does not have full rank, then Lemma \ref{kernel} implies that the first column of $Y_1$ is~0, so by above $Z_{j,1}=-\nu$, a contradiction. Thus $a\le b_2, c\le b_1$ and, by Lemma \ref{kernel}, $Y_1$ and $Y_2$ are as described in (\ref{eq.Z_0})
with $\mu_1\neq 0$, $\nu_1\neq 0$. Since $Z_{a,1}=0$, we infer $\mu_1=\nu_1$.
\end{proof}

Given integers $a,b\geq 1$ and $\alpha\in F$ we consider matrices $f(\alpha),g(\alpha),h(\alpha)\in M_{a,b}$ of respective forms
$$
\left(\begin{array}{cccc}
                      0 & \dots & 0 & \alpha \\
                      0 & \dots & 0 & 0 \\
                      \vdots & \vdots & \vdots & \vdots \\
                      0 & \dots & 0 & 0\\
                    \end{array}
                  \right), \left(\begin{array}{cccc}
                      0 & \dots & 0 & * \\
                      \vdots & \vdots & \vdots & \vdots \\
                      0 & \dots & 0 & * \\
                      0 & \dots & 0 & \alpha\\
                    \end{array}
                  \right),\left(\begin{array}{cccc}
                      \alpha & * & \dots & *  \\
                      0 & \dots & \dots & 0 \\
                      \vdots & \vdots & \vdots & \vdots \\
                      0 & \dots & \dots & 0\\
                    \end{array}
                  \right),
$$
where the entries $*$ will play no role whatsoever.

\begin{prop}\label{reduccion}
Given $\al\in F$ and a sequence $(d_1,d_2,d_3,d_4)$ of positive integers, let $\h$ be
the subalgebra of $\gl(d)$, $d=d_1+d_2+d_3+d_4$, generated by $A$ and $X$, where

-- $A\in\gl(d)$ is the $0$-diagonal block matrix
\[
A= J^{d_1}(\alpha)\oplus J^{d_2}(\alpha-\lambda)\oplus J^{d_3}(\alpha-2\lambda)\oplus J^{d_4}(\alpha-3\lambda),
\]

-- $X\in\gl(d)$ is a $1$-diagonal block matrix whose blocks (1,2), (2,3), (3,4) satisfy
\[
 X(1,2)_{d_1,1}=X(2,3)_{d_2,1}=X(3,4)_{d_3,1}=1.
\]
Then $Y(1,4)=0$ for all $Y\in \h$ if and only if $(d_1,d_2,d_3,d_4)=(1,1,1,1)$.
\end{prop}
\begin{proof} {\sc Sufficiency.} Suppose $(d_1,d_2,d_3,d_4)=(1,1,1,1)$. Then $[A,X]=\la X$, so  $Y(1,3)=Y(2,4)=Y(1,4)=0$ for all $Y\in\h$.

\smallskip

\noindent {\sc Necessity.} Suppose $Y(1,4)=0$ for all $Y\in\h$. Given $(i,j)$, $1\leq i<j\leq 4$, we set
$$
D_{i,j}=(-1)^{d_j-1}{{d_i+d_j-2}\choose{d_i-1}}.
$$
Let
$$m=\max\{d_1+d_2,d_2+d_3,d_3+d_4\}\text{ and }Z=(\ad_{\gl(d)} A-\la 1_{\gl(d)})^{m-2}(X)\in\h.
$$
Then $Z$ is a 1-diagonal block matrix, where
$$
Z(1,2)=\delta_{m,d_1+d_2}f(D_{1,2}),\; Z(2,3)=\delta_{m,d_2+d_3}f(D_{2,3}),\;Z(3,4)=\delta_{m,d_3+d_4}f(D_{3,4}).
$$
Set $U=[X,Z]$. Then $U$ is a 2-diagonal block matrix, where
$$
U(1,3)=\delta_{m,d_2+d_3}g(D_{2,3})-\delta_{m,d_1+d_2}h(D_{1,2}),
$$
$$
U(2,4)=\delta_{m,d_3+d_4}g(D_{3,4})-\delta_{m,d_2+d_3}h(D_{2,3}).
$$
Note that $U=0$ if and only if $(d_1,d_2,d_3,d_4)=(1,1,1,1)$. Suppose, if possible, that $(d_1,d_2,d_3,d_4)\neq (1,1,1,1)$.
Choose $k$ as large as possible such that $V=(\ad_{\gl(d)} A-\la 1_{\gl(d)})^k(U)\neq 0$. By hypothesis, $[X,V]=0$,
so Lemma \ref{lemma1} implies $\text{rank}\,V(1,3)=d_1\le d_3$ and $\text{rank}\,V(2,4)=d_4\le d_2$ (*). Several cases arise:

\smallskip

\noindent{\it Case 1.} $d_1+d_2=d_2+d_3>d_3+d_4$. We have $d_1=d_3$,
$d_4=\mathrm{rank}\,V(2,4)=1$ and $d_1=\mathrm{rank}\,V(1,3)\leq 2$.
From $d_4=1$ we infer $V=U$. Whether $d_1=1$ or $d_1=2$, we readily see
that the condition $\mu_1=\nu_1$ from Lemma \ref{lemma1} is violated.

\smallskip

\noindent{\it Case 1'.} $d_3+d_4=d_2+d_3>d_1+d_2$. This is dual to Case 1, and hence impossible.

\medskip

\noindent{\it Case 2.} $d_1+d_2=d_3+d_4>d_2+d_3$. Then $d_1>d_3$ and $d_4>d_2$, contradicting (*).

\smallskip

\noindent{\it Case 3.} $d_1+d_2>d_2+d_3,d_3+d_4$.  Then $d_1>d_3$, contradicting (*).

\smallskip

\noindent{\it Case 3'.} $d_3+d_4>d_2+d_3,d_1+d_2$. Then $d_4>d_2$, contradicting (*).

\smallskip

\noindent{\it Case 4.} $d_2+d_3>d_1+d_2,d_3+d_4$. In this case, $d_1=\mathrm{rank}\, V(1,3)=1$
and $d_4=\mathrm{rank}\,V(2,4)=1$, whence $V=U$.
We readily see that the condition $\mu_1=\nu_1$ from Lemma \ref{lemma1} is violated.

\smallskip

\noindent{\it Case 5.} $d_1+d_2=d_2+d_3=d_3+d_4$. We have $d_3=d_1=\mathrm{rank}\, V(1,3)\leq 2$ as well as $d_2=d_4=\mathrm{rank}\,V(2,4)\leq 2$. If $(d_1,d_2)=(2,2)$ then $k=1$ and thus $\mathrm{rank}\, V(1,3)=1=\mathrm{rank}\,V(2,4)$, contradicting (*). Whether $(d_1,d_2)=(2,1)$ or $(d_1,d_2)=(1,2)$, we have $V=U$, and we readily see that the condition $\mu_1=\nu_1$ from Lemma \ref{lemma1} is violated.
\end{proof}

\section{Classifying the relatively faithful uniserial representations of $\g$}\label{sec3}

We assume throughout this section that $\g=\langle x\rangle \ltimes L(V)$, where $x$ acts on $V$ via a single lower Jordan
block $J_n(\la)$, $n>1$, relative to a basis $v_0,\dots,v_{n-1}$ of $V$.

\begin{prop}\label{ensu} Suppose $\lambda\neq 0$ and let $T:\g\to\gl(U)$ be a relatively faithful uniserial
representation of dimension $d$.
Then there is a basis $\B$ of $U$, an integer $\ell>2$, a sequence of positive integers $(d_1,\dots,d_\ell)$
satisfying $d_1+\cdots+d_\ell=d$, and a scalar $\al\in F$,
such that the matrix representation $R:\g\to\gl(d)$ associated to $\B$ is normalized standard relative
to $(\ell, (d_1,\dots,d_\ell),\al)$.
\end{prop}

\begin{proof} Noting that $[\g,\g]=V\oplus \Lambda^2(V)$ and $[[\g,\g],[\g,\g]]=\Lambda^2(V)$,
the proof of \cite[Theorem 3.2]{CPS} applies almost verbatim to yield the desired result.
\end{proof}

\begin{theorem}\label{main1} Suppose $\lambda\neq 0$. Then $\g$ has no relatively
faithful uniserial representations of length $>3$.
\end{theorem}

\begin{proof} Let $T:\g\to\gl(U)$ be a relatively faithful representation. By Proposition \ref{ensu},
there is a basis $\B$ of $U$, an integer $\ell>2$, a sequence of positive integers $(d_1,\dots,d_\ell)$
satisfying $d_1+\cdots+d_\ell=d$,
and a scalar $\al\in F$ such that the matrix representation $R:\g\to\gl(d)$ associated to $\B$
is normalized standard relative to $(\ell, (d_1,\dots,d_\ell),\al)$.

Suppose, if possible, that $\ell>3$.  By Lemma \ref{funk}, there is some $i$ such that $d_i+d_{i+1}=n+1$.
Since $\ell>3$, we may consider the representation of $\g$, say $S$, obtained from $R$ by choosing any set of four
contiguous indices taken from $\{1,\dots,\ell\}$ including $i$ and $i+1$. Then $\ker(S)\cap V=(0)$ by Lemma \ref{funk}.
Moreover, $\Lambda^2(V)$ is not contained in $\ker(S)$ because $S$ involves a non-zero 2-diagonal block matrix, as indicated in the proof of Proposition \ref{reduccion}.

We may thus assume without loss of generality that $\ell=4$ and $(d_1,d_2,d_3,d_4)\neq (1,1,1,1)$.
Since $R$ is a representation and $\Lambda^2 V$ commutes with $V$,
it follows from the shape of the matrices in $R(\g)$ that block $(1,4)$
of $R(x)$ is zero for all $x\in\g$, which contradicts Proposition \ref{reduccion}.
\end{proof}

\begin{theorem}\label{main2} Suppose $\lambda\neq 0$. Then every
relatively faithful uniserial representation of $\g$ is isomorphic to one and
only one normalized representation $R_{a,b,c,M,N,\al}$ of non-extreme type.
\end{theorem}

\begin{proof} Let $T:\g\to\gl(U)$ be a relatively faithful representation of dimension $d$. By Proposition \ref{ensu}, there
is a basis $\B$ of $U$, an integer $\ell>2$, a sequence of positive integers $(d_1,\dots,d_\ell)$ satisfying
$d_1+\cdots+d_\ell=d$, and a scalar $\al\in F$ such that the matrix representation $R:\g\to\gl(d)$ associated
to $\B$ is normalized standard relative to $(\ell, (d_1,\dots,d_\ell),\al)$.

Theorem \ref{main1} gives $\ell=3$.
Set $(a,b,c)=(d_1,d_2,d_3)$. We have $a+b\leq n+1$ and $b+c\leq n+1$, with equality holding in at least one case, by Lemma \ref{funk}.
Thus $a+b=n+1$ and $c\leq a$, or $b+c=n+1$ and $a\leq c$. It follows that $R$ is isomorphic to $R_{a,b,c,M,N,\al}$, where $M$ and $N$
are the blocks in the first superdiagonal of $R(v_0)$, and $R_{a,b,c,M,N,\al}$ is of non-extreme type by Proposition \ref{normal2}.
Uniqueness follows from Proposition \ref{normal2}.
\end{proof}

\section{Further cases}

We assume throughout this section that $\g=\langle x\rangle \ltimes L(V)$, where $x\in\GL(V)$.
When the Jordan decomposition of $x$ acting on $V$ has more than one block, other representations are possible. As an illustration, let $m,n\geq 1$, let $\la,\mu\in F$ (we allow the case $\la=\mu$), and suppose $v_0,\dots,v_{n-1},w_0,\dots,w_{m-1}$ is a basis of $V$ relative to which
$$
[x,v_0]=\la v_0+v_1, [x,v_1]=\la v_1+v_2,\dots, [x,v_{n-1}]=\la v_{n-1},
$$
$$
 [x,w_0]=\mu w_0+w_1, [x,w_1]=\mu w_1+w_2,\dots, [x,w_{m-1}]=\mu w_{m-1}.
$$
Let $(a,b,c)$ be a triple of positive integers satisfying
$$a+b=n+1,\; b+c=m+1,
$$
suppose $M\in M_{a\times b}$ and $N\in M_{b\times c}$ satisfy $M_{a,1}\neq 0$ and $N_{b,1}\neq 0$, and let $\al\in F$.
We may then define the uniserial representation $S=S_{a,b,c,M,N,\al}:\g\to\gl(d)$, $d=a+b+c$, as follows:
$$
S(x)=A=\left(
               \begin{array}{ccc}
                 J^a(\al) & 0 & 0 \\
                 0 & J^b(\al-\la) & 0 \\
                 0 & 0 & J^c(\al-\la-\mu) \\
               \end{array}
             \right),
$$
$$
S(v_k)=(\ad_{\gl(d)} A-\la 1_{\gl(d)})^k \left(
               \begin{array}{ccc}
                 0 & M & 0 \\
                 0 & 0 & 0 \\
                 0 & 0 & 0 \\
               \end{array}
             \right),\quad 0\leq k\leq n-1,
$$
$$
S(w_k)=(\ad_{\gl(d)} A-\la 1_{\gl(d)})^k \left(
               \begin{array}{ccc}
                 0 & 0 & 0 \\
                 0 & 0 & N \\
                 0 & 0 & 0 \\
               \end{array}
             \right),\quad 0\leq k\leq m-1.
$$
The fact that $a+b=n+1$ and $b+c=m+1$, together with \cite[Proposition 2.2]{CPS}, ensure that $\ker(S)\cap V=(0)$. Moreover,
since $S(v_0\wedge w_{b-1})\neq 0$, it follows that $\Lambda^2(V)$ is not contained in $\ker(S)$. Thus, $S$ is relatively faithful.

We may imbed $\g$ as a subalgebra of $\g'=\langle x'\rangle\ltimes  L(V')$, where $x'$ has Jordan decomposition
$$
J_{n_1}(\la)\oplus\cdots\oplus J_{n_e}(\la)\oplus J_{m_1}(\mu)\oplus\cdots\oplus J_{m_f}(\mu),
$$
where
$$
n=n_1\geq\dots \geq n_e,\quad m=m_1\geq \dots \geq m_f,
$$
$$
n_2\leq n-2,\; n_3\leq n-4,\; n_4\leq n-6,\dots,\; n_e\leq n-2(e-1),
$$
$$
m_2\leq m-2,\; m_3\leq m-4,\; m_4\leq m-6,\dots,\; m_f\leq m-2(f-1),
$$
$$
e\leq \mathrm{min}\{a,b\},\; f\leq \mathrm{min}\{b,c\}.
$$
Then \cite[Theorem 4.1]{CPS} ensures that we may extend the above representation $S$ of $\g$
to a uniserial representation $S'$ of $\g'$ in such that a way that we still have $\ker(S')\cap V'=(0)$.
Since $\Lambda^2(V)$ is not contained in $\ker(S)$, it follows automatically that $\Lambda^2(V')$ is not contained in $\ker(S')$.
Thus, $S'$ is also relatively faithful.

If $n>1$ (resp. $m>1$) then $S$ (and therefore $S'$) is not faithful, as all wedges $v_i\wedge v_j$ (resp. $w_i\wedge w_j$)
are in the kernel of $S$.

The case $n=1$ and $m=1$ leads to the representation $S_{\al}:\g\to\gl(3)$, given by
$$
x\mapsto \left(
           \begin{array}{ccc}
             \alpha & 0 & 0 \\
             0 & \alpha-\la & 0 \\
             0 & 0 & \alpha-\la-\mu \\
           \end{array}
         \right),
$$
$$
v_0\mapsto \left(
           \begin{array}{ccc}
             0 & 1 & 0 \\
             0 & 0 & 0 \\
             0 & 0 & 0 \\
           \end{array}
         \right),
            w_0\mapsto \left(
           \begin{array}{ccc}
             0 & 0 & 0 \\
             0 & 0 & 1 \\
             0 & 0 & 0 \\
           \end{array}
         \right), v_0\wedge w_0 \mapsto \left(
           \begin{array}{ccc}
             0 & 0 & 1 \\
             0 & 0 & 0 \\
             0 & 0 & 0 \\
           \end{array}
         \right).
$$
This is a faithful uniserial representation.

Suppose next that $x$ acts diagonalizably on $V$, as in the preceding example.
Depending on the nature of the eigenvalues of $x$, there may be other examples of relatively faithful uniserial representations.
Indeed, let $\g=\langle x\rangle\ltimes L(V)$,
where $n>1$, $\la\in F$ and $v_1,\dots,v_n$ is a basis of $V$ such that $$xv_1=i_1\la v_1,xv_2=i_2\la v_2,\dots,xv_n=i_n\la v_n,$$
for positive integers $1=i_1<i_2<\dots<i_n$. Setting $p=i_n+2$ and $J=J^p(0)$, we may then define the uniserial representation $T:\g\to\gl(p)$,
as follows:
$$
x\mapsto \mathrm{diag}(\al,\al-\la,\dots,\al-(p-1)\la),
$$
$$
v_1\mapsto J^{i_1},\; v_2\mapsto J^{i_2},\dots,v_{n-1}\mapsto J^{i_{n-1}}, v_n\mapsto \be E^{1,p-1}+\ga E^{2,p}.
$$
Here we require $\be\neq\ga$ to ensure that $\Lambda^2(V)$ is not contained in $\ker(T)$. Since $\ker(T)\cap V=(0)$,
it follows that $T$ is relatively faithful. Note that $T$ is only faithful when $n=2$.




\begin{thebibliography}{RBMW}

\bibitem[C1]{C1} P. Casati, \emph{The classification of the perfect cyclic
$\mathfrak{sl}(n+1)\ltimes \mathbb{C}^{n+1}$--modules}, Journal of Algebra (2017), http://dx.doi.org/10.1016/j.jalgebra.2016.11.035.

\bibitem[C2]{C2} P. Casati, \emph{Irreducible SL$_{n+1}$ Representations remain
indecomposable restricted to some abelian aubalgebras}, J. Lie
Theory, 20 (2010) 393--407.

\bibitem[CGS]{CGS} L. Cagliero, L. Gutierrez and F. Szechtman, \emph{Classification of finite dimensional uniserial representations of conformal Galilei algebras}, Journal of Mathematical Physics 57 (2016) 101706.

\bibitem[CPS]{CPS} P. Casati, A. Previtali and F. Szechtman,
\emph{Indecomposable modules of solvable Lie algebras}, arXiv:1702.00084.

\bibitem[CS1]{CS1} L. Cagliero and F. Szechtman,
\emph{On the theorem of the primitive element with applications to the representation
theory of associative and Lie algebras},
Canad. Math. Bull. 57 (2014) 735--748.

\bibitem[CS2]{CS2} L. Cagliero and F. Szechtman,
\emph{Indecomposable modules of 2-step solvable Lie algebras in arbitrary characteristic},
Communications in Algebra 44 (2016) 1--10.

\bibitem[CS3]{CS3} L. Cagliero and F. Szechtman, \emph{The classification of uniserial $\mathfrak{sl}(2)\ltimes  V (m)$--modules and
a new interpretation of the Racah–-Wigner $6j$--symbol}, Journal of Algebra 386 (2013) 142-–175.

\bibitem[CMS]{CMS}  P. Casati, S. Minniti and V. Salari,
\emph{Indecomposable representations of the Diamond Lie algebra}
J. Math. Phys., 51 (2010) 033515, 20 pp.

\bibitem[DdG]{DdG} A. Douglas and H. de Guise, \emph{Some nonunitary, indecomposable representations of the Euclidean algebra $\mathfrak{e}(3)$},
 J. Phys. A: Math. Theor. 43 (2010) 085204.

\bibitem[DR]{DR} A. Douglas and J. Repka, \emph{Indecomposable
representations of the Euclidean algebra ${\mathfrak e}(3)$ from irreducible
representations of $\sl(4)$}, Bull. Aust. Math. Soc., 83 (2011)
439--449.

\bibitem[J]{J} H. P. Jakobsen, \emph{Indecomposable
finite-dimensional representations of a class of Lie algebras and
Lie superalgebras}, Supersymmetry in mathematics and physics,
125–-138, Lecture Notes in Math., 2027, Springer, Heidelberg,
2011.


\bibitem[GP]{GP} I.M. Gelfand, V.A. Ponomarev, \emph{Remarks on the classification of a
pair of commuting linear transformations in a finite dimensional
vector space}, Functional Anal. Appl., 3 (1969) 325-326.

\bibitem[M]{M} I. Makedonskyi, \emph{On wild and tame finite-dimensinal Lie algebras},
Funct. Anal. Appl., 47 (2013) 271--283.

\end{thebibliography}
\end{document}